\documentclass[11pt]{amsart}

\usepackage{amsmath}
\usepackage{amssymb}
\usepackage{bm}
\usepackage{graphicx}
\usepackage{psfrag}
\usepackage{color}
\usepackage{xcolor}
\usepackage{listings}

\definecolor{keywordcolor}{rgb}{0.2, 0.2, 0.75}
\definecolor{stringcolor}{rgb}{0.0, 0.5, 0.0}
\definecolor{commentcolor}{rgb}{0.5, 0.5, 0.5}
\definecolor{backgroundcolor}{rgb}{0.95, 0.95, 0.95}
\definecolor{numbercolor}{rgb}{0.3, 0.3, 0.3}

\lstdefinestyle{pythonstyle}{
    language=Python,
    backgroundcolor=\color{backgroundcolor},
    basicstyle=\ttfamily\small,
    keywordstyle=\color{keywordcolor}\bfseries,
    stringstyle=\color{stringcolor},
    commentstyle=\color{commentcolor}\itshape,
    numberstyle=\tiny\color{numbercolor},
    numbers=left,
    numbersep=10pt,
    showstringspaces=false,
    breaklines=true,
    frame=single,
    tabsize=4,
    captionpos=b
}

\definecolor{red}{rgb}{1,0,0}
\definecolor{felix}{rgb}{1,0,0}

\definecolor{pink}{rgb}{1,0,0.4}
\definecolor{marly}{rgb}{1,0,0.4}

\definecolor{blue}{rgb}{0,0,1}
\definecolor{joseph}{rgb}{0,0,1}

\definecolor{anna}{HTML}{d595fc}

\definecolor{arav}{rgb}{0.037,0.43,0.631}

\definecolor{bryan}{HTML}{fc9221}

\definecolor{jason}{HTML}{4a8e1f}

\definecolor{jiya}{HTML}{4c2dbd}

\usepackage{hyperref}
\hypersetup{colorlinks=true, linkcolor=blue, citecolor=magenta, urlcolor=green}
\usepackage{url}
\usepackage{algpseudocode}
\usepackage{fancyhdr}
\usepackage{mathtools}
\usepackage{tikz-cd}
\usepackage{xy}
\input xy
\xyoption{all}
\usepackage{stmaryrd}
\usepackage{calrsfs}
\usepackage{cleveref}

\newtheorem{theorem}{Theorem}[section]
\newtheorem{lemma}[theorem]{Lemma}
\newtheorem{proposition}[theorem]{Proposition}
\newtheorem{corollary}[theorem]{Corollary}

\theoremstyle{definition}
\newtheorem{definition}[theorem]{Definition}
\newtheorem{example}[theorem]{Example}

\theoremstyle{remark}

\numberwithin{equation}{section}

\newcommand{\aaa}{\mathbb{A}}
\newcommand{\cc}{\mathbb{C}}

\newcommand{\nn}{\mathbb{N}}
\newcommand{\N}{\mathbb{N}}
\newcommand{\pp}{\mathbb{P}}

\newcommand{\qq}{\mathbb{Q}}
\newcommand{\Q}{\mathbb{Q}}
\newcommand{\rr}{\mathbb{R}}

\newcommand{\zz}{\mathbb{Z}}
\newcommand{\Z}{\mathbb{Z}}
\newcommand{\nni}{\nn_0 \cup \{\infty\}}

\providecommand\ldb{\llbracket}
\providecommand\rdb{\rrbracket}

\newcommand{\supp}{\text{supp}\,}

\newcommand{\uu}{\mathcal{U}}



\begin{document}
	
\mbox{}
\title{On the set of atoms and strong atoms in additive monoids of cyclic semidomains}

\author{Jiya Dani}
\address{CMI\\MIT\\Cambridge, MA 02139}
\email{cs.program2004@gmail.com}

\author{Anna Deng}
\address{CMI\\MIT\\Cambridge, MA 02139}
\email{annadeng08@gmail.com}

 \author{Marly Gotti}
 \address{Apple Inc.\\One Apple Park Way\\Cupertino, CA 95014}
 \email{marlygotti@apple.com}

\author{Bryan Li}
\address{CMI\\MIT\\Cambridge, MA 02139}
\email{wowo2888@gmail.com}

\author{Arav Paladiya}
\address{CMI\\MIT\\Cambridge, MA 02139}
\email{aravpaladiya1@gmail.com}

\author{Joseph Vulakh}
\address{Department of Mathematics\\MIT\\Cambridge, MA 02139}
\email{jvulakh@mit.edu}

\author{Jason Zeng}
\address{CMI\\MIT\\Cambridge, MA 02139}
\email{jasonzeng124@gmail.com}

\date{August 15, 2025}

\begin{abstract}
	Let $M$ be a cancellative and commutative monoid. A non-invertible element of $M$ is called an atom (or irreducible element) if it cannot be factored into two non-invertible elements, while an atom $a$ of $M$ is called strong if $a^n$ has a unique factorization in $M$ for every $n \in \nn$. The monoid $M$ is atomic if every non-invertible element factors into finitely many atoms (repetitions allowed). For an algebraic number $\alpha$, we let $M_\alpha$ denote the additive monoid of the subsemiring $\nn_0[\alpha]$ of $\cc$. The atomic structure of $M_\alpha$ reflects intricate interactions between algebraic number theory and additive semigroup theory. For $m, n \in \nn_0 \cup \{ \infty \}$ (with $m \le n$), the pair $(m,n)$ is called realizable if there exists an algebraic number $\alpha \in \cc$ such that $M_\alpha$ has $m$ strong atoms and $n$ atoms. Our primary goal is to identify classes of realizable pairs with the long-term goal of providing a complete description of the full set of realizable pairs.
\end{abstract}

\bigskip
\maketitle


\bigskip
\section{Introduction}
\label{sec:intro}

The study of factorizations in algebraic structures such as commutative monoids, rings, and semirings has become a vibrant area of modern algebra, intersecting with combinatorics~\cite{AGL24}, number theory~\cite{GK24}, and geometry~\cite{dG22}. A particularly rich framework for these investigations is provided by additive monoids arising from commutative semidomains~\cite{CGG20}, especially the so-called simple extensions $\mathbb{N}_0[\alpha]$ of the semidomain $\nn_0$ for all $\alpha \in \aaa$, where $\aaa$ denotes the set of all algebraic numbers (see~\cite{ABLST23,CG22} and references therein). For each $\alpha \in \aaa$, we denote by $M_\alpha$ the underlying additive monoid of the semidomain $\nn_0[\alpha]$. The atomic structure of~$M_\alpha$ is tightly connected to the arithmetic and algebraic properties of the minimal polynomial of~$\alpha$, denoted~$m_\alpha(x)$ throughout this paper.

\smallskip

A nonzero element of $M_\alpha$ is called an \emph{atom} if it cannot be expressed in $M_\alpha$ as a sum of two nonzero elements. We let $\mathcal{A}(M_\alpha)$ denote the set of all atoms of~$M_\alpha$. The monoid $M_\alpha$ is said to be \emph{atomic} if every nonzero element can be written as a sum of finitely many atoms (repetitions allowed). It was shown in~\cite[Theorem~4.2]{CG22} that $M_\alpha$ is atomic if and only if~$1$ is an atom; in this case, $\mathcal{A}(M_\alpha) = \{1, \alpha, \dots, \alpha^n \}$ for some $n \in \nn_0$ when $M_\alpha$ is finitely generated while $\mathcal{A}(M_\alpha) = \{\alpha^n : n \in \nn_0\}$ when $M_\alpha$ is not finitely generated.

\smallskip

An atom $a$ in $M_\alpha$ is called a \emph{strong atom} if, for each $n \in \nn$, the only decomposition of $na$ as a finite sum of atoms in $M_\alpha$ is the obvious one, which consists of $n$ copies of $a$. The set of all strong atoms of $M_\alpha$ is denoted by $\mathcal{S}(M_\alpha)$. It follows from the definition of a strong atom that $\mathcal{S}(M_\alpha) \subseteq \mathcal{A}(M_\alpha)$.

\smallskip

A nonzero $p \in M_\alpha$ is defined to be a \emph{prime} if for all $b,c \in M_\alpha$, the condition $b+c \in p + M_\alpha$ implies that either $b \in p + M_\alpha$ or $c \in p + M_\alpha$. The set of all primes of $M_\alpha$ is denoted by $\mathcal{P}(M_\alpha)$. Clearly, every prime in $M_\alpha$ is a strong atom. It was established in~\cite{AGS25} (CrowdMath 2022) that, for each positive algebraic number $\alpha \in \aaa$, the monoid~$M_\alpha$ contains prime elements if and only if it is a unique factorization monoid (UFM). In such a case, one has
\[
    \mathcal{P}(M_\alpha) = \mathcal{S}(M_\alpha) = \mathcal{A}(M_\alpha) = \{\alpha^n : n \in \ldb 0, d-1\rdb \},
\]
where $d$ is the degree of the minimal polynomial of~$\alpha$. However, as we will see in Example~\ref{ex:one strong atom} and Theorem~\ref{thm:strong atoms for irreducible n-th roots}, the monoid~$M_\alpha$ may contain strong atoms even when it is not a UFM.

\smallskip

More generally, the notions of atoms, strong atoms, and primes can be defined for any commutative cancellative monoid~$M$ analogously to the way we have just defined these notions in~$M_\alpha$. For such a monoid~$M$, we obtain the natural inclusions
\begin{equation} \label{eq:inclusion for primes, s-atoms, atoms}
    \mathcal{P}(M) \subseteq \mathcal{S}(M) \subseteq \mathcal{A}(M),
\end{equation}
where $\mathcal{P}(M)$, $\mathcal{S}(M)$, and $\mathcal{A}(M)$ denote the sets of primes, strong atoms, and atoms of~$M$, respectively. For an integral domain~$R$, we let $R^*$ denote its multiplicative monoid (the multiplicative monoid consisting of all nonzero elements of $R$). In the recent paper~\cite{FFNSW24}, Fadinger et al. study the chain of inclusions~\eqref{eq:inclusion for primes, s-atoms, atoms} in the setting of multiplicative monoids of integral domains, providing explicit examples of atomic integral domains that exhibit each of the eight possible combinations regarding the existence or absence of primes, strong atoms that are not atoms, and atoms.

\smallskip

The first systematic study of the atomic and factorization-theoretic structure of the family of monoids~$M_\alpha$ was carried out by Correa-Morris and Gotti in~\cite{CG22}. Motivated by their investigation, we undertake a detailed analysis of the sets of atoms and the sets of strong atoms in the monoids~$M_\alpha$ with the aim of determining the possible pairs
\[
    (|\mathcal{S}(M_\alpha))|, |\mathcal{A}(M_\alpha)|)
\]
as $\alpha$ varies over $\mathbb{A}$. Since we have already observed that $M_\alpha$ is a UFM precisely when it contains primes, we focus our attention on the sets $\mathcal{A}(M_\alpha)$ and $\mathcal{S}(M_\alpha)$ and disregard $\mathcal{P}(M_\alpha)$ in our analysis. We say that a pair $(m,n) \in \nn_0^2$ is \emph{realizable} if it belongs to the set
\[
     \big\{ (|\mathcal{S}(M_\alpha)|, |\mathcal{A}(M_\alpha)|) : \alpha \in \aaa \big\}.
\]

\smallskip

In this paper, we try to understand which pairs of $\nn_0^2$ are realizable. Towards this end, we introduce and analyze the function $f \colon \mathbb{A} \to (\mathbb{N}_0 \cup \{\infty\})^2$ defined as follows:
\[
    f(\alpha) := (|\mathcal{S}(M_\alpha)|, |\mathcal{A}(M_\alpha)|)
\]
for all $\alpha \in \aaa$. The function $f$ describes the number of strong atoms and atoms of the associated monoid $M_\alpha$. We identify various realizable and non-realizable pairs, provide constructive methods to produce algebraic numbers $\alpha$ for which $f(\alpha)$ assumes prescribed values, and examine how transformations of minimal polynomials (such as the substitution $x \mapsto x^k$) affect atomic and factorization-theoretic properties. Our results shed new light on the algebraic and combinatorial behavior of the monoids~$M_\alpha$ and establish connections to classical results, including Descartes' Rule of Signs and Curtiss' Theorem.

\smallskip

This paper is organized as follows. In Section~\ref{sec:background}, we introduce the notation, terminology, and known results that we shall use throughout the paper. Then in Section~\ref{sec:pair realizable by degree-2 algebraic numbers}, we determine the pairs of $\nn_0^2$ that are realizable by the algebraic numbers of degree~$2$, namely, those algebraic numbers whose minimal polynomials have degree~$2$. Finally, in Section~\ref{sec:further realizable pairs}, we first prove that for every $(k,c) \in \nn \times \nn_0$, the pair $(4k+c, 5k+c)$ is realizable, and we then prove that the pair $(n,n+1)$ is realizable if and only if $n \ge 4$. This last result also illustrates that not every pair $(m,n) \in \nn_0^2$ with $m \le n$ is a realizable pair.

\bigskip


\section{Background}
\label{sec:background}

\subsection{General Notation}

Throughout this paper, we let $\zz$, $\qq$, $\rr$, $\aaa$, and $\cc$ stand, as usual, for the set integers, rational numbers, real numbers, algebraic complex numbers, and complex numbers, respectively. We let $\pp$, $\nn$, and $\nn_0$ denote the set of standard primes, positive integers, and nonnegative integers, respectively. 
For any $r,s \in \rr$, we let $\ldb r,s \rdb$ denote the set of integers between $r$ and $s$, that is,
\[
	\ldb r,s \rdb := \{n \in \zz : r \le n \le s\}.
\]
Observe that $\ldb r,s \rdb$ is empty for any $r,s \in \rr$ with $r \geq s$. For a subset $S$ of the real line, we set $S_{\ge r} := \{s \in S : s \ge r\}$ and $S_{> r} := \{s \in S : s > r\}$.

\medskip

\subsection{Commutative Monoids}

Although a monoid is conventionally defined to be a semigroup with an identity element, in the scope of this paper we will reserve the term \emph{monoid} to refer to a semigroup that has an identity element and is cancellative and commutative. Monoids here are written additively unless we explicitly state otherwise.
\smallskip

Let $M$ be a monoid. The set $\uu(M)$ consisting of all invertible elements of the monoids~$M$ is an abelian group that is often referred to as the \emph{group of units} of~$M$. When the group $\uu(M)$ is trivial, the monoid $M$ is called \emph{reduced}. The monoids we study in this paper are reduced. The quotient monoid $M/\uu(M)$ is called the \emph{reduced monoid} of $M$ and is denoted by $M_\text{red}$: the monoid $M_{\text{red}}$ is clearly reduced, and $M \simeq M_{\text{red}}$ if and only if $M$ is already a reduced monoid.
\smallskip

An element $a \in M \! \setminus \! \uu(M)$ is called an \emph{atom} if whenever $a = u+v$ for some $u,v \in M$, then either $u \in \uu(M)$ or $v \in \uu(M)$.  The set consisting of all the atoms of $M$ is denoted by $\mathcal{A}(M)$. Following Coykendall, Dobbs, and Mullins~\cite{CDM99}, we say that $M$ is \emph{antimatter} provided that $\mathcal{A}(M)$ is the empty set. An element of $M$ is said to be \emph{atomic} if it is invertible or can be written as a sum of finitely many atoms in $M$ (repetitions allowed). Following Cohn~\cite{pC68}, we say that the monoid~$M$ is \emph{atomic} if every element of $M$ is atomic. 
\smallskip

For the rest of this section, assume that $M$ is an atomic monoid. An atom $a \in \mathcal{A}(M)$ is called a \emph{strong atom} provided that, for every $n \in \nn$, the element $na$ can be written as a sum of finitely many atoms in only one way, the obvious one. We let $\mathcal{S}(M)$ denote the set of strong atoms of~$M$. A non-invertible element $p \in M$ is called a \emph{prime} if for any $b,c \in M$ the fact that $b+c \in p + M$ implies that $b \in p + M$ or $c \in p + M$. We let $\mathcal{P}(M)$ denote the set consisting of all primes of $M$. It is well known and easy to prove that
\[
    \mathcal{P}(M) \subseteq \mathcal{S}(M) \subseteq \mathcal{A}(M).
\]

Let $\mathsf{Z}(M)$ denote the free commutative monoid on the set of atoms $\mathcal{A}(M_{\text{red}})$, and let $\pi \colon \mathsf{Z}(M) \to M_\text{red}$ denote the only monoid homomorphism that fixes every element of the set $\mathcal{A}(M_{\text{red}})$. The elements of $\mathsf{Z}(M)$ are called \emph{factorizations}. Given a factorization $z := a_1 \cdots a_\ell \in \mathsf{Z}(M)$ for atoms $a_1, \dots, a_\ell \in \mathcal{A}(M_{\text{red}})$, we say that the \emph{length} of~$z$ is $\ell$ and we often denote the length of $z$ by~$|z|$. For each $b \in M$, we set
\[
	\mathsf{Z}_M(b) := \pi^{-1} (b + \uu(M)) \quad  \text{and}  \quad  \mathsf{L}_M(b) := \{ |z| : z \in \mathsf{Z}(b) \}.
\]
When there seems to be no risk of ambiguity, we write $\mathsf{Z}(b)$ and $\mathsf{L}(b)$ instead of $\mathsf{Z}_M(b)$ and $\mathsf{L}_M(b)$, respectively. The monoid $M$ is called a \emph{unique factorization monoid} (resp., \emph{finite factorization monoid}) if $|\mathsf{Z}(b)| = 1$ (resp., $1 \le |\mathsf{Z}(b)| < \infty$) for every $b \in M$. We will write UFM (resp., FFM) instead of unique factorization monoid (resp., finite factorization monoid). Every UFM is clearly an FFM. We say that~$M$ is a \emph{bounded factorization monoid} (BFM) if $1 \le |\mathsf{L}(b)| < \infty$ for all $b \in M$. It follows directly from the corresponding definitions that every FFM is a BFM.

\medskip

\subsection{Polynomial Rings and Semirings}

Let $R$ be an integral domain. Let $f(x)$ be a nonzero polynomial in $R[x]$ and write $f(x) = \sum_{n=0}^d c_n x^n$ for some coefficients $c_0, \dots, c_d \in R$ such that $c_d \neq 0$. Then the \emph{support} of $f(x)$ is the set
\[
	\text{supp} \, f(x) := \{k \in \ldb 0,d \rdb : c_k \neq 0\},
\]
while $\deg f(x) := d$ and $\text{ord} \, f(x) := \min \text{supp} \, f(x)$ are called the \emph{degree} and \emph{order} of $f(x)$, respectively. Descartes' Rule of Signs will be helpful later: it states that the number of variations of sign of a nonzero polynomial $f(x) \in \rr[x]$ has the same parity as and is at least the number of positive roots of $f(x)$ (counting multiplicity). The following theorem is another tool we will use later.

\begin{theorem} \cite[Section~5]{dC18} \label{thm:Curtiss' result}  
	For each nonzero polynomial $f(x) \in \rr[x]$, there exists a nonzero polynomial $\mu(x) \in \rr[x]$ such that the number of variations of sign of $\mu(x)f(x)$ equals the number of positive roots of $f(x)$, counting multiplicity.
\end{theorem}

Now assume that $f(x) = \sum_{n=0}^d c_n x^n \in \qq[x]$ is a nonzero polynomial with $\deg f(x) = d$. Then there exists a unique $r \in \qq_{> 0}$ such that $r f(x)$ is a polynomial in $\zz[x]$ with content~$1$ (i.e., the greatest common divisor of the set of coefficients of $rf(x)$ is~$1$). Then there exists a unique pair of polynomials $(p(x),q(x)) \in \nn_0[x]$ such that $rf(x) = p(x) - q(x)$, called the minimal pair of $f(x)$. Let $\alpha$ be a nonzero algebraic number. We often denote the minimal polynomial of $\alpha$ by $m_\alpha(x) \in \qq[x]$. We refer to the minimal pair of $m_\alpha(x)$ also as the \emph{minimal pair} of~$\alpha$. Recall that the \emph{degree} of $\alpha$ is $\deg m_\alpha(x)$ while the \emph{conjugates} of~$\alpha$ are the roots of $m_\alpha(x)$.

\medskip

\subsection{The Additive Monoid of the Semidomain $\nn_0[\alpha]$}

For any nonempty set $S$ and binary operations `$+$' and `$\cdot$' on $S$, we say that the triple $(S,+,\cdot)$ is a \emph{semidomain} if the following conditions hold:
\begin{itemize}
    \item $(S,+)$ is a monoid,
    \smallskip
    
    \item $(S \setminus \{0\}, \cdot)$ is a monoid, and
    \smallskip

    \item the distributive law holds: $r(s+t) = rs + rt$ for all $r,s,t \in S$.
\end{itemize}

Let $(S,+ ,\cdot)$ be a semidomain, which we denote simply by~$S$. The identity element~$0$ (resp.,~$1$) of $(S,+)$ (resp., $(S \setminus \{0\},\cdot)$) is called the \emph{zero element} (resp., \emph{identity element}) of~$S$. A submonoid $S'$ of $S$ is called a \emph{sub-semidomain} of~$S$ provided that $1 \in S'$ and $S'$ is closed under multiplication. It is well known that any semidomain can be naturally extended to an integral domain (i.e., a commutative ring with identity and without nonzero zero-divisors). 
As for polynomial domains, we let $S[x]$ denote the semidomain consisting of all polynomials over~$S$. Then, for any $\tau \in \cc$, it follows that
\[
    \nn_0[\tau] := \big\{ p(\tau) : p(x) \in \nn_0[x] \big\}.
\]
is a semidomain, which we call the \emph{cyclic semidomain} generated by $\tau$. When $\tau$ is transcendental, $\nn_0[\tau]$ is isomorphic to the polynomial semiring $\nn_0[x]$ and so its additive monoid is isomorphic to the free commutative monoid $\bigoplus_{n \in \nn_0} \nn_0$. As every free commutative monoid is a UFM, the additive monoids of cyclic semidomains generated by transcendental numbers are rather trivial from the viewpoints of atomicity and factorizations. Thus, without loss of generality, we can restrict our attention to cyclic semidomains $\nn_0[\alpha]$ parameterized by algebraic numbers $\alpha \in \mathbb{A}$. As the underlying additive monoids of the semidomains $\nn_0[\alpha]$ (for any $\alpha \in \aaa$) are the central algebraic objects we study in this paper, in order to distinguish the former from the later, it is convenient to introduce the following notation. 
\smallskip

\noindent \textbf{Notation.} For any $\alpha \in \aaa$, we let $M_\alpha$ denote the underlying additive monoid of the cyclic semidomain~$\nn_0[\alpha]$.
\smallskip

For both self-containment and future reference, we conclude this section mentioning some known results on the atomicity and factorization of the monoids~$M_\alpha$ (for any $\alpha \in \aaa$). We start with a characterization of atomicity.

\begin{theorem}\label{thm:atomicity of M alpha}\cite[Theorem~4.2]{CG22} 
    For $\alpha \in \aaa$, the additive monoid $M_\alpha$ is atomic if and only if $1 \in \mathcal{A}(M_\alpha)$, in which case, there exists $\sigma \in \nn_0 \cup \{ \infty \}$ such that
    \begin{equation}
        \mathcal{A}(M_\alpha) = \{\alpha^n : n \in [0, \sigma)  \cap \nn_0 \}.
    \end{equation}
    Moreover, $\sigma = \min \{ n \in \nn : \alpha^n \in \{\alpha^j : j \in \ldb 0, n-1 \rdb \} \}$ if  $M_\alpha$ is finitely generated (and so atomic).
\end{theorem}

Here we have a sufficient condition for the atomicity of the studied monoids.

\begin{proposition}\label{prop:sufficient condition for atomicity}\cite[Proposition~4.5]{AGS25}
    Let $\alpha$ be a positive algebraic number with minimal polynomial polynomial $m_\alpha(X)$. The monoid $M_\alpha$ is atomic with infinitely many atoms if $m_{\alpha}(x)$ has a positive root different from~$\alpha$.
\end{proposition}

We conclude with a characterization of the monoids $M_\alpha$ that are UFMs.

\begin{theorem}\cite[Theorem~5.4]{CG22}  \label{thm:factoriality characterization}
    Let $\alpha$ be an algebraic number with degree $d$, minimal polynomial $m_\alpha(x) \in \qq[x]$, and minimal pair $(p(x), q(x)) \in \nn_0[x]^2$. Then the following conditions are equivalent.
        \begin{enumerate}
            \item[(a)] $M_\alpha$ is a UFM.
            \smallskip
            
            \item[(b)] $\deg m_\alpha(x) = |\mathcal{A}(M_\alpha)|$.
            \smallskip
            
            \item[(c)] $p(x) = x^d$ for some $d \in \N$.
        \end{enumerate}
        Moreover, if $M_\alpha$ is a UFM, then it is finitely generated.
\end{theorem}


\bigskip
\section{Realization of the Sizes of the Sets of Atoms and Strong Atoms}

\medskip

\subsection{Atoms and Strong Atoms} 

Fix an algebraic number $\alpha$. It was proved in~\cite{AGS25} (CrowdMath 2022) that we can take $m \in \nn_0 \cup \{\infty\}$ such that the set of strong atoms of $M_\alpha$ is
\begin{equation}
    \mathcal{S}(M_\alpha) = \big\{ \alpha^k : k \in [0,m) \cap \nn_0 \big\}.
\end{equation}
Thus, when $M_\alpha$ has infinitely many strong atoms, $\mathcal{S}(M_\alpha) = \mathcal{A}(M_\alpha)$. This is not necessarily the case when $M_\alpha$ has only finitely many strong atoms. The following simple example (also worked out in CrowdMath 2022) shows a monoid~$M_\alpha$ with infinitely many atoms and only one strong atom.

\begin{example} \label{ex:one strong atom}
    Consider the polynomial $m(x) := x^2 - 3x + 1 \in \zz[x]$. Observe that $m(x)$ is an irreducible polynomial over $\qq$ such that $m(0) = 1$ and $m(1) = -1$. Thus, $m(x)$ has two positive roots: one root $\alpha \in (0,1)$ and another root $\beta > 1$. By \Cref{prop:sufficient condition for atomicity}, $M_\alpha$ is atomic with infinitely many atoms, so $\mathcal{A}(M_\alpha) = \{\alpha^n : n \in \nn_0\}$. Since $3 \alpha = 1 + \alpha^2$, it follows that $\alpha$ is not a strong atom, and so the inclusion $\alpha M_\alpha \subseteq M_\alpha$ ensures that $\alpha^n$ cannot be a strong atom for any $n \ge 1$. On the other hand, we claim that $1$ is a strong atom. For some $n \in \nn$ and $c_1, \dots, c_N \in \nn_0$, set $g(x) := \big( \sum_{i=1}^N c_i x^i \big) - n$, and observe that $g(x)$ has only one variation of sign, so by Descartes' Rule of Signs, $g(x)$ has one positive real root. However, if $g(\alpha)=0$ such that $n \cdot 1 = \sum_{i=1}^N c_i \alpha^i$, we would need $g(x)$ to be divisible by $m(x)$ and have at least two distinct positive real roots. Thus, $\mathcal{S}(M_\alpha) = \{1\}$.
\end{example}

As our primary purpose is to determine the pairs $(m,n) \in (\nn_0^2 \cup \{\infty\})^2$ for which there exists $\alpha \in \aaa$ such that the monoid $M_\alpha$ has exactly~$m$ strong atoms and~$n$ atoms, it is convenient to formally introduce the following function:
\begin{equation} \label{eq:realization function f}
    f \colon \aaa \to \big\{ (m,n) \in (\nni)^2 : m \le n \big\},
\end{equation}
defined as
\[
    f(\alpha) := \big( |\mathcal{S}(M_\alpha)|, |\mathcal{A}(M_\alpha)|\big)
\]
for all $\alpha \in \aaa$. We are mostly interested here in determining the image of the function~$f$. Thus, let us introduce the following terminology.

\begin{definition}
	We call the function $f$ in~\eqref{eq:realization function f} the \emph{realization function} and, for any $m,n \in \nn_0 \cup \{\infty\}$, we call the pair $(m,n)$ a \emph{realizable pair} if $(m,n) \in f(\aaa)$.
\end{definition}

Observe that in light of Example~\ref{ex:one strong atom}, there exists $\alpha \in \aaa$ such that $f(\alpha) = (1,\infty)$. As a result of the main theorem of the next section, we obtain infinitely many $\alpha \in \aaa$ such that $f(\alpha) = (0,\infty)$.

The following lemma will be helpful later in characterizing the number of strong atoms of the monoid~$M_\alpha$.

\begin{lemma} \label{lem:strong atom as sum of lower degree terms}
    Let $\alpha$ be an algebraic number with minimal polynomial $m_\alpha(x)$. If~$M_\alpha$ is finitely generated and $1 \le |\mathcal{S}(M_\alpha)| < |\mathcal{A}(M_\alpha)|$, then there exists a polynomial $f(x) \in \qq[x]$ of degree $|\mathcal{S}(M_\alpha)| - \deg m_\alpha(x)$ so that $f(x)m_\alpha(x)$ has negative leading coefficient and all other coefficients nonnegative.
\end{lemma}

\begin{proof}
    Set $m := |\mathcal{S}(M_\alpha)|$ and $n := |\mathcal{A}(M_\alpha)|$. It suffices to show that the equality $\alpha^m = \sum_{i=0}^{m-1} q_i\alpha^i$ holds for some $q_0, \dots, q_{m-1} \in \Q_{\ge 0}$. Indeed, if $\ell$ is a common multiple of the denominators $\mathsf{d}(q_0), \dots, \mathsf{d}(q_{m-1})$, then the polynomial
    \[
        \ell x^m - \sum_{i=0}^{m-1} \ell q_i x^i \in \zz[x]
    \]
    will have~$\alpha$ as a root and, therefore, be divisible by $m_\alpha(x)$. Observe that because $m \le n-1$, the atom $\alpha^{n-1}$ is not a strong atom, so we can write $\alpha^{n-1}$ as $\sum_{i=0}^{n-2} q_{n-1,i}\alpha^i$, where each $q_{n-1,i}\in\qq_{\ge 0}$.
    
    Now we prove by induction that for all $k \in \nn$ such that $k \in \ldb m,n-1 \rdb$, we can write $\alpha^k$ as a linear combination of the elements $\alpha^i$ for $i \in \ldb 0,k-1 \rdb$. We have proved the base case of $k = n - 1$, so fix $m \leq k < n - 1$ and assume that for all $j$ where $k< j<n$, we have $\alpha^j$ equal to $\sum_{i=0}^{j-1}q_{j,i}\alpha^i$, where $q_{j,i}\in \qq_{\ge 0}$. Denote this summation with $z_j$. Because $\alpha$ is not a root of any polynomial in $\qq[x]$ with degree $m \ge 1$, at least two coefficients $q_{j,i}$ are nonzero for each~$j$. Now let $\sum_{i=0}^{n-1}c_i\alpha^i$ be a nontrivial factorization of $c\alpha^k$ for some large enough $c\in\N$, which must exist because $\alpha^k$ is not a strong atom. Going from the $\alpha^{n-1}$ term to the $\alpha^{k+1}$ term, we can successively rewrite the $\alpha^i$ in each term as $z_i$ for $k< i<n$ from our assumption. We then obtain a new representation $c\alpha^k=\sum_{i=0}^k q_{k,i}\alpha^i$, where, since each $z_i$ has at least two nonzero coefficients, at least two $q_{k,i}$ are nonzero. We can then assume that $q_{k,k}=0$ by subtracting an appropriate rational multiple of $\alpha^k$ from both sides. This completes the induction step. Hence we can write $\alpha^m = \sum_{i=0}^{m-1} q_i\alpha^i$, as desired.
\end{proof}

As direct consequences of the previous proposition, we obtain the following corollaries, which will be useful later.

\begin{corollary}\label{cor:minimal poly condition for atoms}
    Let $\alpha$ be an algebraic number with minimal polynomial $m_\alpha(x) \in \qq[x]$. Then the following statements hold.
    \begin{enumerate}
        \item If $p(x)$ is a polynomial of minimal degree such that $p(x)m_\alpha(x) = x^n - \sum_{i=0}^{n-1} a_i x^i$ for some $n \in \N$ and coefficients $a_0, \dots, a_{n-1} \in \N_0$, then $|\mathcal{A}(M_\alpha)| = n$.
        \smallskip
        
        \item If $M_\alpha$ is finitely generated and $q(x)$ is a polynomial of minimal degree such that  $q(x)m_\alpha(x) = b_m x^m - \sum_{i=0}^{m-1} b_i x^i$ for some $m \in \N$ and coefficients $b_0, \dots, b_m \in \N_0$, then $|\mathcal{S}(M_\alpha)| = m$.
    \end{enumerate}
\end{corollary}

\begin{corollary}\label{cor:more strong atoms than degree}
    Let $\alpha$ be an algebraic number with minimal polynomial $m_\alpha(x) \in \qq[x]$. If $M_\alpha$ is finitely generated, then $|\mathcal{S}(M_\alpha)| \ge \deg m_\alpha(x)$.
\end{corollary}

\medskip

\subsection{Binomial Minimal Polynomials}

It was proved in~\cite[Theorem~6.2]{GG18} that, for each $q \in \qq_{> 0} \setminus \nn^{\pm 1}$, the monoid $M_q$ is atomic with $|\mathcal{A}(M_q)| = \infty$. This result was generalized in~\cite[Proposition~3.9]{ABLST23} for any positive irreducible $n$-th root of any $q \in \qq_{> 0} \setminus \nn^{\pm 1}$ (for $q \in \qq_{> 0}$, an \emph{irreducible root} of $q$ is an $n$-th root $\alpha := \sqrt[n]{q}$ of $q$ such that $\alpha, \dots, \alpha^{n-1}$ are all irrationals). It is not difficult to argue that the polynomial $x^n - q$ is irreducible in $\qq[x]$ for any pair $(n,q) \in \zz_{\ge 2} \times \qq_{> 0}$ such that $\sqrt[n]{q}$ is an irreducible $n$-th root of $q$ (see \cite[page 297]{sL02}). Let us establish a similar result for the set of strong atoms.

\begin{theorem} \label{thm:strong atoms for irreducible n-th roots}
    For $q \in \qq_{> 0} \setminus \nn^{\pm 1}$, let $\alpha$ be a positive irreducible $n$-th root~$q$. Then the monoid $M_\alpha$ is atomic with $\mathcal{A}(M_\alpha) = \{\alpha^n : n \in \nn_0 \}$ and $\mathcal{S}(M_\alpha) = \emptyset$.
\end{theorem}

\begin{proof}
    As shown in \cite[page 297]{sL02}, if $q=\frac{a}{b}$ for relatively prime $a,b\in \mathbb{Z}_{\ge2}$, the polynomial $x^n-\frac{a}{b}$ is irreducible in $\qq[x]$. Thus, the polynomial $bx^n-a$ is irreducible in $\zz[x]$. From~\cite[Proposition~3.9]{ABLST23}, the set of atoms is $\mathcal{A}(M_\alpha) = \{\alpha^n : n \in \nn_0\}$, and so $M_\alpha$ is an atomic monoid. Now observe that, for each $k \in \nn_0$, the equality $a \cdot \alpha^k = b \cdot \alpha^{n+k}$ holds, and so the atom $\alpha^k \in M_\alpha$ cannot be a strong atom of $M_\alpha$. As a consequence, we conclude that the set of strong atoms of $M_\alpha$ is empty.
\end{proof}

\begin{corollary}
    There exists $\alpha \in \aaa$ such that $f(\alpha) = (0,\infty)$.
\end{corollary}

In the following example we exhibit a class of monoids $M_\alpha$ with infinitely many strong atoms, which implies that the pair $(\infty, \infty)$ is realizable.

\begin{proposition}
     Let $\alpha$ be a positive algebraic number with exactly three positive conjugates (including $\alpha$ itself). Then $f(\alpha) = (\infty, \infty)$.
\end{proposition}

\begin{proof}
    Define $m(x)$ as the minimal polynomial of $\alpha$. Since $m(x)$ has $3$ positive roots, $M_\alpha$ is atomic with infinitely many atoms by \Cref{prop:sufficient condition for atomicity}. Assume, for the sake of contradiction, that some $\alpha^\ell$ is not a strong atom in $M_\alpha$. Then $\alpha$ must be a root of some polynomial $q(x) = \sum_{i=0}^N c_i x^i$, where $c_\ell \le -1$, and every other $c_i$ is nonnegative. Observe that $q(x)$ has at most $2$ changes in sign, so it has at most two positive real roots by Descartes' Rule of Signs. However, this contradicts that $q(x)$ is a multiple of $m(x)$, which has three positive real roots. Thus, every power of $\alpha$ is a strong atom, and so $f(\alpha) = (\infty, \infty)$.
\end{proof}


\bigskip
\section{Degree-2 Minimal Polynomials}
\label{sec:pair realizable by degree-2 algebraic numbers}

Let $\aaa_2$ be the subset of $\cc$ consisting of all algebraic numbers with minimal polynomials of degree~$2$. In this section, we determine $f(\aaa_2)$ by considering cases regarding the form of minimal polynomials of degree~$2$.

\begin{proposition} \label{prop:deg-2 minimal polynomial}
    Let $\alpha$ be an algebraic number of degree $2$, and let $m(x)$ be the only primitive polynomial that is a scalar multiple of the minimal polynomial of $\alpha$. Then the following statements hold.
    \begin{enumerate}
    	\item If $m(x) \in \nn_0[x]$, then $f(\alpha) = (0,0)$.
    	\smallskip
    	
        \item If $m(x) := ax^2+bx-c \in \zz[x]$ for some $a,b,c\in \mathbb{N}$ with $\sqrt{b^2+4ac}\in \mathbb{R}\setminus \mathbb{Q}$, then $f(\alpha) \in \{(0,0), (0,\infty)\}$.
        \smallskip

        \item If $m(x) := ax^2-bx+c \in \zz[x]$ for some $a,b,c \in \mathbb{N}$ with $\sqrt{b^2-4ac} \in \mathbb{R} \setminus \mathbb{Q}$, then $f(\alpha) = (1,\infty)$.
        \smallskip

        \item If $m(x) := ax^2-bx-c \in \zz[x]$ for some $a,b,c\in \mathbb{N}$ with $\sqrt{b^2+4ac} \in \mathbb{R} \setminus \mathbb{Q}$, then $f(\alpha) \in \{(2,2), (2,\infty)\}$.
    \end{enumerate}
\end{proposition}

\begin{proof}
	(1)  If $m(x) \in \nn_0[x]$, then $\nn_0[\alpha] = \zz^2 \times \zz/a$,
    where $a$ is the leading coefficient of $m(x)$.
    In this case, $M_\alpha$ is an abelian group, and thus an antimatter monoid. As a consequence, $\mathcal{A}(M_\alpha)$ and $\mathcal{S}(M_\alpha)$ are both empty sets, which means that $f(\alpha) = (0,0)$.
	\smallskip
	
    (2) If $c=1$, then $1=ax^2+bx$ and so $1\in M_\alpha$ is not an atom. Consequently, no elements of the generating set are atoms. In this case, $M_\alpha$ is antimatter, so $f(\alpha) = (0,0)$. Now assume that $c>1$. Then $c\cdot 1$ can be written as $a\cdot \alpha^2+b\cdot \alpha$, and so the inclusion $1 M_\alpha \subseteq M_\alpha$ ensures that $\alpha^n$ cannot be a strong atom for any $n \ge 0$. Hence $M_\alpha$ contains no strong atoms. Let us argue now that the equality $\mathcal{A}(M_\alpha) = \{\alpha^n : n \in \nn_0 \}$ holds. 

    Suppose, for the sake of contradiction, that $M_\alpha$ has finitely many atoms, and then set $n := |\mathcal{A}(M_\alpha)|$. Therefore we can write $\alpha^n = \sum_{i=0}^{n-1}p_i\alpha^i$ for some coefficients $p_0, \dots, p_{n-1} \in\mathbb{N}_0$. This assumption implies that $m(x)$ divides $x^n-\sum_{i=0}^{n-1}p_ix^i$ in $\zz[x]$, which in turn implies that $a=1$. Let $x^{n-2}+\sum_{i=0}^{n-3} q_ix^i$ for $q_0, \dots, q_{n-3}\in\mathbb{Z}$ be the quotient of the polynomial division such that
    \[
        x^n-\sum_{i=0}^{n-1}p_ix^i = (x^2+bx-c)\left(x^{n-2}+\sum_{i=0}^{n-3}q_ix^i\right).
    \]
    We can assume, without loss of generality, that $q_0 \neq 0$ as, otherwise, we can divide out a factor of~$x$. Since $p_0\geq 0$, we must have $cq_0=p_0$, meaning that $q_0$ is positive. Since $p_1\geq 0$, we have $cq_1-bq_0>0$, which also implies that $q_1$ is positive. Furthermore, we have the following relations: 
    \begin{align*}
       p_2 = cq_2-bq_1-q_0&\geq 0 \\
       p_3 = cq_3-bq_2-q_1&\geq 0 \\ 
       \vdots\\
       p_{n-3} = cq_{n-3}-bq_{n-4}-q_{n-5}&\geq 0.
    \end{align*}
    If $q_i$ and $q_{i+1}$ are positive, and we have the relation $cq_{i+2}-bq_{i+1}-q_i\geq 0$, then $q_{i+2}$ must be positive. Hence, by induction we have $q_{n-3}>0$. However, we also have the relation $p_{n-1}=-b-q_{n-3}\geq 0$, which implies $q_{n-3}<0$. This gives a contradiction. Therefore $f(\alpha) = (0,\infty)$ in this case.
    \smallskip
    
    (3) If $b=1$, then it is impossible for $\sqrt{b^2-4ac}$ to be real. Now assume that $b>1$. Then the roots of $m(x)$ are $\frac1{2a}(b \pm \sqrt{b^2-4ac})$ which are both positive. By \Cref{prop:sufficient condition for atomicity}, $M_\alpha$ is atomic with infinitely many atoms. Because $b\cdot \alpha = a\cdot \alpha^2+c$, $\alpha$ is not a strong atom, and the inclusion $\alpha M_\alpha \subseteq M_\alpha$ ensures that $\alpha^n$ cannot be a strong atom for any $n \ge 1$. On the other hand, we claim that $1$ is a strong atom. Suppose that this is not the case, and take $n \in \nn$ and $c_1, \dots, c_N \in \nn_0$ such that $n \cdot 1 = \sum_{i=1}^N c_i \alpha^i$. Now set
    \[
    	g(x) := \bigg( \sum_{i=1}^N c_i x^i \bigg) - n,
    \]
    and observe that $g(x)$ has at least two distinct positive real roots because it is divisible by $m(x)$, which has two distinct positive real roots. However, this contradicts the statement of the Descartes' Rule of Signs as $g(x)$ has only one variation of sign. Hence $\mathcal{S}(M_\alpha) = \{1\}$, so $f(\alpha) = (1, \infty)$.
    \smallskip
    
    (4) Since $a\cdot \alpha^2 = b\cdot \alpha +c$, the inclusion $\alpha^2 M_\alpha \subseteq M_\alpha$ ensures that $\alpha^i$ cannot be a strong atom for any $i \ge 2$. We claim that $\mathcal{S}(M_\alpha)=\{1,\alpha\}$. Suppose that $\alpha$ is not a strong atom. Then $m(x)$ divides a polynomial $f(x)$ of the form $c_0 - c_1x + \sum_{i=2}^n c_ix^i$ for some $n \in \nn$ and $c_i \in \nn_0$.
    Let $\sum_{i=0}^{N-2} d_i x^i$ for $d_0$, \dots, $d_{N - 2} \in \Z$
    be the quotient of the polynomial division such that
    \[
        c_0 - c_1x + \sum_{i = 2}^n c_ix^i
        = (ax^2 - bx - c)\sum_{i = 0}^{n - 2} d_ix^i.
    \]
    Since $-cd_0 = c_0 > 0$, we must have $d_0 < 0$.
    Also, $-c_1 = -bd_0 - cd_1$,
    and since $-c_1 < 0$ and $-bd_0, c > 0$, we must have $d_1 > 0$.
    We also have
    \begin{align*}
        c_2 = -cd_2 - bd_1 + ad_0 &\geq 0 \\
        c_2 = -cd_3 - bd_2 + ad_1 &\geq 0 \\
        \vdots\\
        c_{n - 2} = -cd_{n - 2} - bd_{n - 3} + ad_{n - 4} &\geq 0.
    \end{align*}
    We claim that for each $i \in \ldb 0, n - 3\rdb$, at least one of $d_i$ and $d_{i + 1}$ is negative.
    The statement is true for $i = 0$, so we proceed by induction. Assume the statement is true for some $i < n - 3$. If $d_{i + 1} < 0$, then the statement automatically is true for $i + 1$. Otherwise, we have $d_i < 0$ and $d_{i + 1} \geq 0$, in which case $c_{i + 2} = -cd_{i + 2} - bd_{i + 1} + ad_i \geq 0$
    and $-bd_{i + 1} \le 0$, $ad_i < 0$, which forces $d_{i + 2} < 0$, completing the induction.
    Now $d_{n - 2}$ must be positive, as $c_n = ad_{n - 2}$. Thus, $d_{n - 3}$ is negative,
    so $c_{n - 1} = ad_{n - 3} - bd_{n - 2} > 0$ gives a contradiction. Therefore $\alpha$ is a strong atom as desired, and so~$1$ is as well, proving that $|\mathcal{S}(M_\alpha)| = 2$.

    We now determine the possible values of $|\mathcal{A}(S)|$.
    If $a=1,$ then $\alpha^2 = b\cdot \alpha + c$ holds, so $\alpha^2$ is not an atom, and the inclusion $\alpha^2 M_\alpha \subseteq M_\alpha$ ensures that $\alpha^n$ cannot be an atom for any $n \ge 2$. The set of atoms is thus $\{1,\alpha\}$, as otherwise $m(x)$ would not be minimal.
    
    Now assume $a>1$. We claim that $\mathcal{A}(M_\alpha) = \{\alpha^n : n \in \nn_0\}$. To prove this assume, for the sake of contradiction, that $\alpha^n$ is not an atom. By Gauss's Lemma, if $m(x)$ divides a polynomial of the form $x^n-\sum_{i=0}^{N-1} c_ix^i$ for some $c_0, \dots, c_{N-1} \in \mathbb{N}_0$, then $a=1$, which is a contradiction. Hence $f(\alpha) \in \{(2,2), (2,\infty) \}$, concluding the proof.
\end{proof}


\bigskip
\section{Further Realizable Pairs}
\label{sec:further realizable pairs}

In this section, we determine other pairs that belong to the image of our main function $f$. Let us start with the following proposition.

\begin{proposition} \label{prop:when $m(x^k)$ is irreducible}
    Let $\alpha$ be an algebraic number with minimal polynomial $m(x) \in \qq[x]$ such that $m(x^k)$ is irreducible in $\qq[x]$ for some $k \in \nn_{\ge 2}$. If $\beta$ is a root of $m(x^k)$, then $f(\beta) = kf(\alpha)$.
\end{proposition}

\begin{proof}
    Fix $k \in \nn_{\ge 2}$ such that the polynomial $m(x^k)$ is irreducible in $\qq[x]$, and let $\beta$ be a root of $m(x^k)$. As $m(x^k)$ is monic and irreducible, it must be the minimal polynomial of $\beta$. We are done once we establish the following two equalities:
    \[
        |\mathcal{A}(M_\beta)| = k |\mathcal{A}(M_\alpha)| \quad \text{ and } \quad |\mathcal{S}(M_\beta)| = k |\mathcal{S}(M_\alpha)|.
    \]
    To argue the first of these two equalities, take $g(x) \in \mathbb{Z}[x]$ to be the monic polynomial with minimal degree such that its non-leading coefficients are negative and $m(x)q(x)=g(x)$ for some $q(x) \in \mathbb{Z}[x]$. Then $\text{deg } g(x) = |\mathcal{A}(M_\alpha)|$. Since $m(x^k)q(x^k)=g(x^k)$ also holds, and $g(x^k)$ satisfies the condition that its leading coefficient is positive while the rest of its coefficients are negative, we have that
    \[
        \text{deg } g(x^k) = k |\mathcal{A}(M_\alpha)| \geq |\mathcal{A}(M_\beta)|.
    \]
    In order to show that equality holds, we will show that $g(x^k)$ has minimal degree among polynomials with this property.
   
    Suppose, for the sake of a contradiction, that the degree of $g(x^k)$ is not minimal among such polynomials, and let $h(x)\in\mathbb{Z}[x]$ be such a polynomial of minimal degree. Then, we can write $m(x^k)b(x) = h(x)$ for some $b(x) \in \mathbb{Z}[x]$ such that $\deg b(x) < \deg q(x^k)$. Let us argue the following claim.
    \smallskip

    \noindent \textsc{Claim.} $k$ is a common divisor of $\text{supp } b(x)$.
    \smallskip

    \noindent \textsc{Proof of Claim.} Suppose that $\text{supp } b(x)$ had two distinct elements mod $k$. For each $e \in \ldb 0,k-1 \rdb$, we let $h_e(x)$ denote the sum of the terms of $h(x)$ with degree $e$ mod $k$. Since $\text{supp } b(x)$ has two distinct elements modulo~$k$, we see that $h_a(x),h_b(x) \neq 0$ for some distinct $a,b$. Since $\text{supp } m(x^k)$ only consists of multiples of~$k$, the polynomial $m(x^k)$ divides both $h_a(x)$ and $h_b(x)$. We know that $h(x)$ only has one positive coefficient, meaning one of $h_a(x)$ and $h_b(x)$ must have all of its coefficients being negative. However, $m(x^k)$ has a positive real root, which gives a contradiction by Descartes' Rule of Signs. Therefore all the positive integers in $\text{supp } b(x)$ are equal modulo~$k$. Since $0 \in \text{supp } b(x)$ for $h(x)$ to be minimal, we see that $\text{supp } b(x) \subset k\nn_0$. The claim is then established.
    \smallskip

    In light of the proved claim, we can write $b(x) = b_1(x^k)$ and $h(x) = h_1(x^k)$ where $b_1(x), h_1(x) \in \mathbb{Z}[x]$. Therefore $m(x^k)b_1(x^k)= h_1(x^k)$, and so $m(x)b_1(x)= h_1(x)$. This, along with the inequality $\deg h(x) < \deg g(x^k)$, implies that $\deg h_1(x) < \deg g(x)$. However, this last inequality contradicts the minimality of $g(x)$. As a consequence, $|\mathcal{A}(M_\beta)| = \deg g(x^k) = k \deg g(x) = k |\mathcal{A}(M_\alpha)|$, as desired.
    \smallskip
   
    Establishing the equality $|\mathcal{S}(M_\beta)| = k |\mathcal{S}(M_\alpha)|$ is quite similar. First, let $g(x)$ be a polynomial in $\mathbb{Z}[x]$ with exactly one negative coefficient such that the degree~$d$ of the term with the negative coefficient is minimal, and $m(x)q(x)=g(x)$ for some $q(x) \in \mathbb{Z}[x]$. Then, for the sake of contradiction, we let $h(x)$ be a polynomial satisfying the same condition for $m(x^k)$, where the degree of the negative term is less than $k\deg g(x)$, and write $m(x^k)b(x) = h(x)$ for some polynomial $b(x) \in \mathbb{Z}[x]$. Since $h(x)$ only has one negative coefficient, we can follow as we did previously to obtain that all the terms of $\text{supp } b(x)$ are divisible by $k$. Then we can write $b(x) = b_1(x^k)$ and $h(x) = h_1(x^k)$, giving $m(x)b_1(x) = h_1(x)$. However, $h_1(x)$ has its negative coefficient term with degree less than $d$, which contradicts the minimality of $g(x)$. Hence the equality $|\mathcal{S}(M_\beta)| = k |\mathcal{S}(M_\alpha)|$ also holds, which concludes our proof.
\end{proof}

Observe that the hypothesis in the statement of Proposition~\ref{prop:when $m(x^k)$ is irreducible} is guaranteed to hold if the minimal polynomial $m(x)$ satisfies Eisenstein's criterion, whence we obtain the following corollary.

\begin{corollary} \label{cor:multiple realizable if eisenstein}
    Let $\alpha$ be an algebraic number with minimal polynomial $m(x) \in \zz[x]$ satisfying the hypothesis of Eisenstein's criterion. Then, for every $k \in \nn$, the polynomial $m(x^k)$ is irreducible in $\qq[x]$ and $f(\beta) = kf(\alpha)$ for any root $\beta$ of $m(x^k)$.
\end{corollary}

As an immediate consequence of Proposition~\ref{prop:when $m(x^k)$ is irreducible}, we also obtain the following realization result.

\begin{corollary}
    If $f(\alpha) = (m,n)$, then the pair $(cm, cn)$ is realizable for every $c \in \nn_0$ such that $m_\alpha(x^c)$ is irreducible. 
\end{corollary}
\medskip

Next, we prove the following realization result.

\begin{theorem} \label{thm:4k_5k_realize}
    The pair $(4k+c,5k+c)$ is realizable for all $(k,c) \in \nn \times \nn_0$.
\end{theorem}

\begin{proof}
    For the case when $c=0$, we claim that for the positive root $\alpha$ of $m(x) = x^3 - 8x^2 + 4x - 2$, we have $f(\alpha) = (4,5)$. Since $m(x)$ satisfies Eisenstein's criterion, it will follow from Corollary~\ref{cor:multiple realizable if eisenstein} that the pair $(4k,5k)$ is realizable for any $k \in \nn$. To see that $|\mathcal{A}(M_\alpha)| = 5$, note that
    \[
    	(x^2+2x+1)m(x) = x^5 - 6x^4 - 11x^3 - 2x^2 - 2,
    \]
    so $|\mathcal{A}(M_\alpha)|\le 5$. By Corollary~\ref{cor:more strong atoms than degree}, $|\mathcal{A}(M_\alpha)| \ge 3$, and clearly $|\mathcal{A}(M_\alpha)|\ne 3$ as $m(x)$ has multiple negative coefficients. It remains to show that $|\mathcal{A}(M_\alpha)| \ne 4$.
    
    Suppose, for the sake of a contradiction, that there are $4$ atoms. Then there must exist some linear polynomial $f(x) = bx+a \in \zz[x]$ such that $f(x)m(x)$ has a leading coefficient $1$ and all other coefficients non-positive. Multiplying, we get that
    \[
        f(x)m(x) = bx^4 + (a-8b)x^3 + (4b-8a)x^2 + (4a-2b)x - 2a.
    \]
    Note that we must have $b=1$, and also $4b-8a\le 0$ and $4a-2b\le 0$ imply that $b-2a=0$. Thus, clearly $a$ is not an integer, so we reach a contradiction, proving that indeed $|\mathcal{A}(M_\alpha)| = 5$. To show that $|\mathcal{S}(M_\alpha)| = 4$, note that due to Corollary~\ref{cor:more strong atoms than degree}, we have $|\mathcal{S}(M_\alpha)|\ge 3$, and clearly there are more than three strong atoms as $m(x)$ has multiple negative coefficients. Note that $(2x+1)m(x) = 2x^4-15x^3-2$, which satisfies the conditions of Corollary~\ref{cor:minimal poly condition for atoms}, so $|\mathcal{S}(M_\alpha)| = 4$.
    \smallskip
    
    We now consider the case when $c>0$. Fix $(k,c) \in \nn \times \N$, and let us show that the $(4k+c,5k+c)$ belongs to the image of $f$. To do so, first observe that the polynomial
    \[
    	m(x) = x^{3k+c} - 8x^{2k+c} + 4x^{k+c} - 2x^c - 2
    \]
    is irreducible in $\zz[x]$ by Eisenstein's criterion at the ideal $2\zz$ and, therefore, in $\qq[x]$ by virtue of Gauss's lemma. In addition, $m(x)$ has a positive root $\alpha$ as $m(0) < 0$. We proceed to show that $M_\alpha$ is an atomic monoid with exactly $5k+c$ atoms and $4k+c$ strong atoms. To see that there are at most $5k+c$ atoms, consider the polynomial $f(x) = 1 + 2x^k + x^{2k}$, so that
    \[
        f(x)m(x)=-2-2x^c-4x^k-2x^{2k}-2x^{2k+c}-11x^{3k+c}-6x^{4k+c}+x^{5k+c}.
    \]
    Similarly, there are at most $4k+c$ strong atoms since multiplying $m(x)$ by $1+2x^k$ gives
    \[
    	(1+2x^{k}) m(x) = -2-2x^c-4x^k-15x^{3k+c}+2x^{4k+c}.
    \]
    Note that there are no issues in both cases when $c$ is a multiple of $k$, as all non-leading coefficients will remain negative. Now suppose for the sake of contradiction that $\alpha^{4k+c-1}$ is not a strong atom. Then there exists some $q(x)$ so that $m(x)q(x)$ has all non-positive coefficients except for the coefficient of $x^{4k+c-1}$. We can assume $\deg q(x)=k-1$ by Lemma~\ref{lem:strong atom as sum of lower degree terms}, so write $q(x)=\sum_{d=0}^{k-1}a_dx^d$. Then
    \[
    	 m(x) a_dx^d = -2a_dx^d-2a_dx^{d+c}+4a_dx^{k+c+d}-8a_dx^{2k+c+d}+a_dx^{3k+c+d}.
    \]
    Observe that for any $d<k-1$, the set $\supp(a_{k-1}x^{k-1}m(x))\,\cap\,\supp(a_{d}x^{d}m(x))$ contains only numbers less than $k+c+d$. Therefore the coefficients of the $x^{2k+c-1}$ and $x^{3k+c-1}$ terms are given by $4a_{k-1}$, and $-8a_{k-1}$, respectively, from which it follows that $a_{k-1} = 0$. However, this contradicts that $\deg q(x) = k-1$.
    
    Similarly, for the atoms, suppose for the sake of a contradiction that $\alpha^{5k+c-1}$ is not an atom. Then there exists some $q(x)=\sum_{d=0}^{k-1}q_d(x)$ where $q_d(x)=a_dx^d+b_dx^{d+k}$ so that $m(x)q(x)$ has all non-positive coefficients except for the term $x^{5k+c-1}$, which has coefficient~$1$. We have
    \begin{align*}
    		m(x)q_d(x) &= -2a_dx^d-2a_dx^{d+c}-2b_dx^{d+k}+(4a_d-2b_d)x^{k+c+d} \\
    							&+(4b_d-8a_d)x^{2k+c+d}+(a_d-8b_d)x^{3k+c+d}+b_dx^{4k+c+d}.
    \end{align*}
	Note that the set $\supp(q_d(x)m(x))\,\cap\,\supp(q_{d'}(x)m(x))$ contains only numbers less than $2k+c+d'$  for each pair $(d',d)$ with $d'<d<k$. Therefore the coefficients of the terms $x^{2k+c+d}, x^{3k+c+d}$, and $x^{4k+c+d}$ in $m(x)q(x)$ are given by $4a_d-8b_d$, $a_d-8b_d$, and $b_d$, respectively. Hence, when $d<k-1$, the inequalities $b_d \le 0$ and $a_d \le 8b_d$ hold, and so it follows that $a_d \le 0$. Now, consider the coefficient of term $x^d$ for $d<k-1$ in $m(x)q(x)$. Only $q_d(x)$ and $q_{d-c}(x)$ can contribute and so the coefficient of the $x^d$-term of the polynomial $q(x)m(x)$ is $-2a_d - 2a_{d-c}$, which results from writing the $x^d$-term as $-2a_dx^d - 2a_{d-c}x^{(d-c) + c}$ (set $a_{d-c}$ to be $0$ when $d-c < 0$). Suppose that $a_{d-c}=0$ for some $d$, and notice that because the coefficient in $m(x)q(x)$ of the term $x^d$ cannot be positive, $a_d \le 0$, and so $a_d = 0$. Since $a_{d-c}=0$ for $d<c$, it follows by induction that $a_d = 0$ for each $d<k-1$. Additionally, from the inequalities $b_d\le 0$ and $8b_d \ge a_d$ we obtain that $b_d = 0$. Thus, $q(x) = q_{k-1}(x) = a_{k-1}x^{k-1}+x^{2k-1}$. Now multiplying $q(x)$ by $m(x)$ gives 
    \begin{align*}
    	m(x)q(x) &=-2a_{k-1}x^{k-1}-2a_{k-1}x^{{k-1}+c}-2x^{{2k-1}}+(4a_{k-1}-2)x^{{2k+c}-1} \\
    					&+(4-8a_{k-1})x^{3k+c-1}+(a_{k-1}-8)x^{4k+c-1}+x^{5k+c-1}.
    \end{align*}
	We get $a_{k-1} \ge 0$ from the first coefficient, and $a_{k-1}\le\frac12$ from the coefficient of the term $x^{2k+c-1}$. As a result, $a_{k-1} = 0$, which contradicts $4 - 8a_{k-1} \le 0$ from the coefficient of the term $x^{3k+c-1}$.
\end{proof}

We conclude this section showing that the pair $(n,n+1)$ is realizable for almost all $n \in \nn_0$.

\begin{proposition} \label{prop:m strong atoms and m+1 atoms}
    The pair $(n,n+1)$ is realizable if and only if $n \ge 4$.
\end{proposition}

\begin{proof}
    It follows from Theorem~\ref{thm:4k_5k_realize} that the pair $(4+c, 5+c)$ is realizable for all $c \in \nn_0$. For $n<4$, we claim that $(n,n+1)$ is not realizable. Consider $n=3$. Suppose that for some $\alpha\in\aaa$ the monoid $M_\alpha$ has $3$ strong atoms and $4$ atoms. Then it follows from Corollary~\ref{cor:more strong atoms than degree} that $\deg \, m(x) \le 3$. Suppose for the sake of a contradiction that $\deg m(x) = 3$. Then by Lemma~\ref{lem:strong atom as sum of lower degree terms}, the leading coefficient of $m(x)$ is positive while the rest of its coefficients are non-positive. If $m(x)$ is monic, then $M_\alpha$ must have $3$ atoms by \cite[Theorem~5.4]{CG22}, and otherwise it must have infinitely many atoms as $m(x)$ cannot divide any polynomial $x^k-\sum_{i=0}^{k-1} b_ix^i$ with nonnegative coefficients $b_0, \dots, b_{k-1}$. Thus, $\deg m(x) \le 2$, so the result of Proposition~\ref{prop:deg-2 minimal polynomial} applies. For $n\in\{1,2\}$, we have $\deg m(x) \le 2$ by Corollary~\ref{cor:more strong atoms than degree}. For $n=0$, note that $1$ must be the only atom, and thus $M_\alpha=\nn_0$, which has one strong atom.
\end{proof}


\bigskip
\section*{Acknowledgments}

While working on this paper, the authors were part of the first CrowdMath Internship (CMI 2025), a semester-long math research program hosted by the MIT Mathematics department. The authors would like to thank Dr.~Felix Gotti for his guidance and support as a CMI research mentor. The authors are grateful to the programs MIT PRIMES and CrowdMath for providing this rewarding research opportunity.


\end{document}